\documentclass[a4paper,11pt]{article}
\usepackage[utf8]{inputenc}
\usepackage[T1]{fontenc}
\usepackage[english]{babel}
\usepackage{lmodern}
\usepackage{amsmath}
\usepackage{amssymb,fontenc,amsthm,amsfonts,amsmath,dsfont}
\usepackage[mathscr]{eucal}
\usepackage[top=1.5cm,bottom=2cm,left=3cm,right=2cm]{geometry}
\usepackage{graphicx,color}
\usepackage{hyperref}
\usepackage{authblk}
\usepackage{amsthm}
\usepackage[inline]{enumitem}
\usepackage[lowercase]{theoremref}
\usepackage{framed}

\theoremstyle{definition}
\newtheorem{mydef}{Definition}[section]
\newtheorem{mycor}[mydef]{Corollary}
\newtheorem{mythm}[mydef]{Theorem}
\newtheorem{myremark}[mydef]{Remark}

\newtheorem{myconjecture}[mydef]{Conjecture}

\newtheorem{myexample2}[mydef]{Example}
\newtheorem{myproposition}[mydef]{Proposition}
\newtheorem{mylemma}[mydef]{Lemma}

\newtheorem{mytable}[mydef]{Table}
\numberwithin{equation}{section}

\title{Some cases of a conjecture on L-functions of twisted Carlitz modules}
\author[1]{S. Ehbauer }
\author[1]{D. Logachev \thanks{ \url{logachev94@gmail.com} (corresponding author)} }
\author[2]{M. Sarraff Nascimento \thanks{\url{msnascimento01@gmail.com};}}

\affil[1]{Departamento de Matem\'atica, Instituto de Ci\^encias Exatas, UFAM Manaus, Brasil;} \affil[2]{Departamento de Matem\'atica, UEA Parintins, Brasil}

\begin{document}
\maketitle
\begin{abstract}
We prove two polynomial identies which are particular cases of a conjecture arising in the theory of L-functions of twisted Carlitz modules. This conjecture is stated in \cite[p. 153, 9.3, 9.4]{GL16} and \cite[p. 5, 0.2.4]{LZ17}.
\end{abstract}


\section{Introduction}
The papers \cite{GL16} and \cite{LZ17} develop a theory of $L$-functions of twists of Carlitz modules. Results of these papers depend on an important conjecture \cite[9.3.]{GL16}, \cite[Conj. 0.2.4.]{LZ17}, present paper Conjecture \ref{Conjecture}. For example, it implies that varieties $X(2,n,m,i)$, objects of research of \cite{GL16} and \cite{LZ17}, do not depend on $n$.
\\

Some particular cases of Conjecture \ref{Conjecture} are already proved in \cite{GL16}. The purpose of the present paper is to prove two more particular cases of Conjecture \ref{Conjecture} announced in \cite{GL16}, 9.13 a, b: \thref{SarrafTheorem} and  \thref{EhbauerTheorem}.
\\

To state Conjecture \ref{Conjecture}, we give here an exposition of the theory of L-functions of Anderson t-motives. We borrow the corresponding part of text of \cite{LZ17}. For more details, the reader can look \cite{G92}, \cite{A00}, \cite{L09}.
\\

The authors are grateful to Darij Grinberg, CSENG School of Mathematics, University of Minnesota, for the formulation and the proof of proposition \ref{DetIdent}.
\medskip

{\bf Exposition of the theory of $L$-functions of the twisted $n$-th tensor powers of Carlitz modules.}

\subsection*{Anderson t-motives}

Let $q$ be a power of a prime number $p$ and $\mathbb{F}_q$ the finite field of order $q$. In analogy to the extensions $\mathbb{Z} \subset \mathbb{Q} \subset \mathbb{R} \subset \mathbb{C}$ we consider extensions starting with the polynomial ring $\mathbb{F}_q[\theta]$ of one independent variable $\theta$ over $\mathbb{F}_q$. The field $\mathbb{F}_q(\theta)$ is its field of rational functions. It is the quotient field analogous to $\mathbb{Q}$. The field of the Laurent series $\mathbb{F}_q((1/\theta))$ is the functional field analogous to $\mathbb{R}$.
\\

By definition, $\mathbb{C}_{\infty}$ is the completion of the algebraic closure of $\mathbb{F}_q((1/\theta))$, it is the analogue field to $\mathbb{C}$.

The Anderson ring $\mathbb{F}_q(\theta)[ t, \tau]$ is the ring of
non-commutative polynomials over $\mathbb{F}_q(\theta)$ satisfying the following relations:

\begin{enumerate*}[leftmargin=1cm, label=\arabic*)]
\item $t\cdot \tau = \tau \cdot t$,\quad \quad
\item $\forall a\in \mathbb{F}_q(\theta)\ \ t\cdot a=a \cdot t$,\quad \quad
\item $\forall a\in \mathbb{F}_q(\theta) \ \ \tau \cdot a = a^q \cdot \tau$. \\
\end{enumerate*}

We need the following (less general than in \cite{A86}) version of the definition of Anderson t-motives $M$ over $ \mathbb{F}_q(\theta)$:
\medskip

\begin{mydef} An Anderson t-motive $M$ is a $\mathbb{F}_q(\theta)[ t, \tau]$-module such that
\begin{enumerate}[noitemsep, label=\arabic*)]
\item $M$ considered as a $\mathbb{F}_q(\theta)[ t]$-module is free of finite dimension $r$;
\item $M$ considered as a $\mathbb{F}_q(\theta)[\tau]$-module is free of finite dimension $n$;
\item There exists $k>0$ such that $(t-\theta)^k M/\tau M=0$.
\end{enumerate}
\end{mydef}
Analogously, we can define an Anderson t-motive over $\mathbb{C}_{\infty}$.

Numbers $r$, resp. $n$ are called the rank, resp. the dimension of $M$.

The Carlitz module $\mathfrak{ C}$ is an Anderson t-motive having $r=n=1$. Therefore we have only one element $e=e_1$ in a basis of $M$ over $\mathbb{F}_q(\theta)[ \tau]$ and $\mathfrak{C}$ is given by the equation $te=\theta e +\tau e$. We have: $e$ is also the only element of a basis of $\mathfrak{C}$ over $\mathbb{F}_q(\theta)[t]$, and the multiplication by $\tau$ is given by the formula $$\tau e=(t-\theta)e$$

We shall consider also the $n$-th tensor power of $\mathfrak C$ over the ring $\mathbb{F}_q(\theta)[t]$, denoted by $\mathfrak C^n$. It has the rank $r=1$. The basis has the only element
$e_n=\underbrace{e\otimes e \otimes e \otimes ... \otimes e}_{\substack{\text{n times}}}$

The action of $\tau$ on $e_n$ is given by the formula
\begin{equation}
\tau e_n=\tau e \otimes \tau e \otimes \ldots \otimes \tau e =(t-\theta)^n e_n
\end{equation}
It is easy to check that this formula really defines an Anderson t-motive of rank $r=1$ and dimension $n$.

A twist of $\mathfrak C^n$ is an Anderson t-motive over $ \mathbb{F}_q(\theta)$ which is isomorphic to $\mathfrak C^n$ over $\mathbb{C}_{\infty}$. It is easy to check that these twists are parameterized by polynomials
\begin{equation}
P=\sum_{i=0}^m a_i\theta^i\in \mathbb{F}_q[\theta]
\label{eq:PolynomP}
\end{equation}

Namely, for this $P$ we denote by $\mathfrak C^n_P$ an Anderson t-motive given by the formula

\begin{equation}
e_n=P(t-\theta)^ne_n \; \; \text{ with } e_n \text{ as earlier.}
\end{equation}

Two such twists $\mathfrak C^n_{P_1}$, $\mathfrak C^n_{P_2}$ are isomorphic over $\mathbb{F}_q(\theta)$ iff $P_{1}/P_{2}\in (\mathbb{F}_q(\theta)^{*})^{(q-1)}$, where $(\mathbb{F}_q(\theta)^{*})^{(q-1)}$ are the $(q-1)$-th powers of nonzero elements of $\mathbb{F}_q(\theta)$.

\begin{mycor}
\label{CorollaryNoTwists}
There are no twists for $q=2$.
\end{mycor}

\subsection*{$L$-functions}

For any Anderson t-motive $M$ over $ \mathbb{F}_q(\theta)$ we can define its $L$-function
$L(M,T)\in \mathbb{F}_q[t][[T]]$, as in \cite{G92} or \cite[p 2603]{L09}, where the letter $u$ is used for the variable $\tau$.

A very simple and explicit definition of $L(M,T)$ is given in \cite[p. 123]{GL16}. There is a general formula for $L(M,T)$ based on a version of the Lefschetz trace formula as in \cite[Section 9]{B12} or \cite{L09} or \cite[p. 127, 3.4]{GL16} or in the original paper \cite{A00}.
In the case $M=\mathfrak{C}^n_P$ an explicit formula for $L(\mathfrak C^n_P,T)$ is given in \cite[p. 126, 3.3.]{GL16}.
\\

Let $P$ be as in equation \eqref{eq:PolynomP}.
We denote by $\bar k=\left[ \frac{m+n}{q-1}\right]$ the  integer part of $\frac{m+n}{q-1}$ and let $\mathcal{ M}(P,n,\bar k)=\mathcal{M}(a_*,n,\bar k)$ be the matrix in $M_{\bar k\times \bar k}(\mathbb{F}_q[t])$ whose $(i,j)$-th entry is defined by the formula

\begin{equation}
\mathcal M(P,n,\bar k)_{i,j}=\sum_{l=0}^n (-1)^l\binom{n}{l}a_{jq-i-l}\ t^{n-l}\quad,
\label{eq:MPolynomialEntry}
\end{equation}
where $a_l=0$ if $l\not\in \{0,\dots, m\}$.

For the particular case $n=1$ we have $\mathcal M(P,1,\bar k)_{i,j}=a_{jq-i}t-a_{jq-i-1}$ and therefore
\begin{equation}
\mathcal M(P,1,\bar k)=\left(\begin{matrix} a_{q-1}t-a_{q-2} &  a_{2q-1}t-a_{2q-2}   &   \dots & a_{\bar kq-1}t-a_{\bar kq-2}  \\ a_{q-2}t-a_{q-3}
 &  a_{2q-2}t-a_{2q-3} &   \dots & a_{\bar kq-2}t-a_{\bar kq-3} \\
a_{q-3}t-a_{q-4} &  a_{2q-3}t-a_{2q-4}  &   \dots & a_{\bar kq-3}t-a_{\bar kq-4} \\
\dots & \dots & \dots  & \dots \\ a_{q-\bar k}t-a_{q-\bar k-1}
 &  a_{2q-\bar k}t-a_{2q-\bar k-1}   &  \dots & a_{\bar kq-\bar k}t-a_{\bar kq-\bar k-1} \end{matrix} \right)
\label{eq:MPolynomn1}
\end{equation}

\begin{mythm}[{\cite[p. 126, 3.3.]{GL16}}]
$$L(\mathfrak C^n_P,T)=\det(I_{\bar k} - \mathcal M(P,n,\bar k)T)$$.
\label{thm}
\end{mythm}

\begin{myremark}
~
\begin{enumerate}[noitemsep]
\item Theorem \ref{thm} implies that $L(\mathfrak C^n_P,T)\in (\mathbb{F}_q[t])[T]$ is of degree $\le \bar k$ (for a general $M$ the $L$-function $L(M,T)$ can be a power series in $T$).
\medskip
\item $\mathfrak M(P,n,k)$ of \cite[(3.1.), (3.2.)]{GL16} is $\mathcal M(P,n,\bar k)^t$. Transposition is not important, because we consider determinants.
\medskip
\item $\mathcal M(P,n,\bar k)^t$ is (up to a non-essential change of indices) a particular case of the matrix from \cite[1.5.]{FP98}.
\medskip
\item Formula for $L(\mathfrak C^n_P,T)$ is concordant with the natural inclusion of the set of polynomials of degree $\le m$ to the set of polynomials of degree $\le m'$, where $m'>m$.
\end{enumerate}
\end{myremark}

\begin{myremark}[Non-trivial part]
~
\\

If $\frac{m+n}{q-1}$ is integer then the last column of $\mathcal M(P,n,\bar k)$ has only one non-zero element, namely its lower element which is equal to $(-1)^na_m$. For this case we denote $k:=\bar k-1=\frac{m+n}{q-1}-1$ and we consider the $k\times k$-submatrix of $\mathcal M(P,n,\bar k)$ formed by elimination of its lowest row and last column. We denote this submatrix by $\mathcal M_{nt}(P,n,k)=\mathcal M_{nt}(a_*,n,k)$ and  the L-function $$L_{nt}(\mathfrak C^n_P,T):=\det(I_{k} - \mathcal M_{nt}(P,n,k)T) \text{ (the non-trivial part).}$$
We have
\begin{equation} \label{DefLfunction}
L(\mathfrak C^n_P,T)=L_{nt}(\mathfrak C^n_P,T)\cdot (1-(-1)^na_mT)
\end{equation}
\end{myremark}

We can consider $a_0,\dots, a_m$ as abstract elements, not necessarily as elements of $\mathbb{F}_q$. Particularly, they can belong to a field $K$ of any characteristic. In this general settings, we denote the non-trivial part $L_{nt}(\mathfrak C^n_P,T)$ of the $L$-function by:
\begin{equation}
CH(\mathcal M_{nt}(a_*,n,k),T):= \det (I_{ k} - \mathcal M_{nt}(P,n, k)T)\in \mathbb{Z}[a_0,\ldots, a_m][t][T]
\end{equation}

We define polynomials $H_{i,j,n}(m)$ as coefficients of $CH(\mathcal M_{nt}(a_*,n,k),T)$, by the following formula:

\begin{equation}\label{CHarMnt}
CH(\mathcal M_{nt}(a_*,n,k),T)=\sum_{i=0}^k \sum_{j=0}^{n(k-i)}H_{i,j,n}(m)t^jT^{k-i} \quad \text{ with }
H_{i,j,n}(m) \in \mathbb{Z} [a_0,\dots, a_m]
\end{equation}
For $l=1, \ldots k$ we define varieties $X(q,n,m,l)\subset \bar K^{m+1}$ ($\bar K$ the algebraic closure of $K$) by
\begin{equation}
X(q,n,m,l) = \{(a_0, \ldots ,a_m) \in \bar K^{m+1} \vert H_{i,j,n}(m) = 0 \text{ for }i = 0, \ldots , l-1 \text{ and all } j
\}
\end{equation}
The polynomials $H_{i,j,n}(m) \in \mathbb{Z} [a_0,\dots, a_m]$ are homogeneous polynomials. By this we can consider $X(q,n,m,l)\subset P^m(\bar K)$ as projective varieties.
Since $H_{k,0,n} =1$ we have $X(q,n,m,m) = \emptyset$.
\\
\medskip

The meaning of the varieties $X(q,n,m,l)\subset P^m(\bar K)$ is the following: Let $K=\mathbb F_q$ and let $P$ be from \ref{eq:PolynomP}. We have \cite[Def. 0.1.11]{LZ17}:
\begin{mydef}
~
\\
The analytic rank of $\mathfrak C^n_P$ at $\infty$ is $\bar{k}- \deg_T L(\mathfrak C^n_P,T)$.
It is denoted by $r_\infty=r_\infty(n,P) = r_\infty(n,a_*)$.
\end{mydef}
\noindent
The equalities \ref{DefLfunction} - \ref{CHarMnt} show, that the set of coefficients $\{a_0, \ldots , a_m\}$, such that $r_\infty(n, a_*) \ge l$, is exactly the set $X(q,n,m,l)(\mathbb{F}_q)$.
\\
\\
In the following we consider exclusively $q=2$ (for $K=\mathbb{C}$ we get a non-trivial theory, although there are no twists over $\mathbb{F}_2$).
\medskip \\
This concludes $k:=m+n-1$.
We extend the definition of $\mathcal M(P,n,\bar k)$ to the value $n=0$. We analogously get a matrix $\mathcal M_{nt}(P,0,m)\in M(m-1\times m-1,\mathbb{Z}[a_0,\ldots, a_m])$ with the entries $$\mathcal M_{nt}(P,0,m)_{i,j} =  a_{2j-i}\quad,$$
which has the form
\begin{equation}
\mathcal M_{nt}(P,0,m)=
\left(\begin{matrix}
a_{1}&  a_{3}&  a_5 & \dots & a_{m-2}&a_m & 0 & 0 & \ldots & 0 \\
a_{0}&  a_{2}&  a_4 & \dots & a_{m-3}&a_{m-1} & 0 & 0 & \ldots & 0\\
0    &  a_{1}&  a_3 & \dots & a_{m-4}&a_{m-2} & a_m & 0 & \ldots & 0\\
0    &  a_0  &  a_2 & \ldots & a_{m-5}& a_{m-3} & a_{m-1} & 0 & \ldots & 0 \\
\dots & \dots & \dots  & \dots & \ldots&\ldots &\ldots & \ldots & \ldots & \ldots  \\
0 & 0 & 0 & \dots & a_1& a_3 & a_5 & a_7  & \ldots & a_m \\
0 & 0 & 0 & \dots & a_0 & a_2 & a_4 & a_6  & \ldots & a_{m-1}
\end{matrix} \right) \quad \text{ for $m$ odd.}
\label{eq:MP0ntmodd}
\end{equation}
(For even $m$ we have a similar form.)
\\
We also have
\begin{equation}
CH(\mathcal M_{nt}(a_*,0,m-1),T):= \det (I_{ m-1} - \mathcal M_{nt}(P,0, m-1)T) \in \mathbb{Z} [a_0,\dots, a_m]
\end{equation}
The entries of the matrix $\mathcal M_{nt}(P,0,m-1)$ do not depend on variable $t$. Therefore
\begin{equation}
CH(\mathcal M_{nt}(a_*,0,m-1),T)=\sum_{i=0}^{m-1}H_{i,0,0}(m)T^{m-1-i} \quad \text{ with }
H_{i,0,0}(m) \in \mathbb{Z} [a_0,\dots, a_m].
\end{equation}
We give the polynomials $H_{i,0,0}(m)$ an extra-notation:
$$D(m,i):=H_{i,0,0} \text{ for } i \in \{0,\ldots, m-1 \}
\quad \text{ and } \quad
D(m,i):=0 \text{ for } i \notin \{0,\ldots, m-1\}.$$
For $l=1, \ldots m$ we define the algebraic variety $X(m,l)$ as the common zeros of the homogeneous polynomials $D(m,0),\dots,D(m,l-1)$.

We have the following conjecture:
\begin{myconjecture}\label{Conjecture}
~

For $q=2$ and for any $i,j,m,n$ there exists $\gamma \in \mathbb{N}$ such that $H_{i,j,n}^\gamma(m)$ is in the ideal $<D(m,0),\dots,D(m,i)>$ generated by
$D(m,0),\dots,D(m,i)$:
\begin{equation}
\label{eq:Conjecture}
H_{ijn}^\gamma(m) \in <D(m,0),\dots,D(m,i)>
\end{equation}
\end{myconjecture}
This conjecture immediately implies that Supp $X(2,n,m,l)$ does not depend on $n$ and is equal to Supp $X(m,l)$, although $X(q,n,m,l)$ and $X(m,l)$ are different as schemes (for example, multiplicities of their irreducible components depend on $n$). This simplifies greatly the study of $X(2,n,m,l)$. Moreover, Conjecture \ref{Conjecture} is used for a description of irreducible components of $X(m,l)$ in \cite{LZ17}. Finally, generalizations of Conjecture \ref{Conjecture} for the case $q>2$ can be used for the study of the behavior of $L(\mathfrak C^n_P,T)$ at $T=1$, see \cite{GL16} and \cite{LZ17} for details.
\\
\\
We give some results already achieved in \cite{GL16}:

\begin{enumerate}
\item We have
\begin{equation}
H_{0,j,1}(m) = \pm a_{j}D(m,0) \text{ for } j=0, \ldots, m \quad ,
\end{equation}
hence $X(m,1)=X(2,1,m,1)$ \ (\cite{GL16}, Theorem III; \cite{LZ17}, 10.5).
\item For any $m,n\ge 1$ and $i\in \{0, \ldots, k\}$ we have (\cite{GL16}, 9.12)
\begin{equation}
H_{i,0,n}(m)=\pm D(m,i-n)\pm a_0D(m,i-n+1)
\end{equation}
\begin{equation}
H_{i,n(k-i),n}(m)=\pm D(m,i-n)\pm a_mD(m,i-n+1)
\end{equation}
\item If Conjecture \ref{Conjecture} holds for $n=1$ then it holds for all $n$ \ (\cite{LZ17}, 7.2.1).
\end{enumerate}

Together with the results of the present paper, apparently these are for $n=1,2$ the only cases where $\gamma$ from Conjecture \ref{Conjecture} is 1. For the case $\gamma >1$ proof of Conjecture \ref{Conjecture} seems to be much more complicated.

\section{Formula for $H_{m-1,1,1}(m)$.}
For $n=1$ we have $k=m$. The matrix $\mathcal M(P,1,k)$ has the form
{
\begin{equation}
\left(\begin{matrix}
a_{1}t - a_0 &  a_{3}t-a_2&  a_5t-a_4 & \dots  & \ldots & 0 \\
a_{0}t&  a_{2}t-a_1&  a_4t-a_3 & \dots & \ldots & 0\\
0    &  a_{1}t-a_0&  a_3t-a_2 & \dots & \ldots & 0 \\
0    &  a_0t  &  a_2t-a_1 & \ldots  & \ldots & 0 \\
\dots & \dots & \dots  & \dots &  \ldots & \ldots   \\
0 & 0 & 0 & \dots  & \ldots & a_m\\
0 & 0 & 0 & \dots  & \ldots & a_mt-a_{m-1}
\end{matrix} \right)
\label{eq:MP1evenred}
\end{equation}}
\begin{mythm} \thlabel{SarrafTheorem}
\begin{equation}
H_{m-2,1,1}(m) = D(m,m-3)-D(m,m-2)^2.
\end{equation}
\end{mythm}

\begin{proof}
For $U=T^{-1}$ we have
\begin{align*}
\det(\mathcal M(P,0,m)-U\cdot I_{m-1})  =
\left(\begin{matrix}
a_{1}-U&  a_{3}&  a_5 & \dots & 0 \\
a_{0}&  a_{2}-U&  a_4 & \dots  & 0\\
0    &  a_{1}&  a_3-U& \dots  & 0\\
0    &  a_0  &  a_2 & \ldots &  0 \\
\dots & \dots & \dots  & \dots  & \ldots  \\
0 & 0 & 0 & \dots &  a_m \\
0 & 0 & 0 & \dots &  a_{m-1}-U
\end{matrix} \right) \\
=D(m,0) + D(m,1)(-U)+ \ldots + D(m,m-1)(-U)^{m-1}
\end{align*}
This gives immediately
\begin{equation*}
D(m,m-2)=a_1+a_2+\ldots +a_{m-1}
\end{equation*}
\begin{equation}
D(m,m-3)=\sum_{1\le j < k \le m-1}
\begin{vmatrix}
a_j & a_{2k-j} \\
a_{2j-k} & a_k
\end{vmatrix}
= \sum_{1\le j < k \le m-1}a_j a_{k} -
a_{2j-k} a_{2k-j}
\end{equation}
It follows, that
\begin{equation}
D(m,m-2)^2=\sum_{1\le i \le m-1} a_i^2 + 2\sum_{1\le j < k \le m-1}a_ja_k
\end{equation}
On the other side we consider the determinant $\det(\mathcal M(P,1,m)- U\cdot I_m)$, which is equal to:
\begin{gather*}
{\footnotesize \begin{vmatrix}
a_{1}t - a_0-U &  a_{3}t-a_2&  a_5t-a_4 & \dots & \ldots & 0 \\
a_{0}t&  a_{2}t-a_1-U&  a_4t-a_3 & \dots &  \ldots & 0\\
0    &  a_{1}t-a_0&  a_3t-a_2-U & \dots  & \ldots & 0 \\
0    &  a_0t  &  a_2t-a_1 & \ldots  & \ldots & 0 \\
\dots & \dots & \dots   & \ldots & \ldots & \ldots   \\
0 & 0 & 0 & \dots  & \ldots & a_m\\
0 & 0 & 0 & \dots  & \ldots & a_mt-a_{m-1} -U
\end{vmatrix}}
\end{gather*}
The coefficient at $(-U)^{m-2}$  has the form:
$$\sum_{1\le j < k \le m}
\begin{vmatrix}
a_jt-a_{j-1} & a_{2k-j}t-a_{2k-j-1}\\
a_{2j-k}t-a_{2j-k-1} & a_kt-a_{k-1}
\end{vmatrix} $$
\begin{gather*}=  \sum_{1\le j < k \le m} (a_ja_k-a_{2j-k}a_{2k-j})t^2 \\
+\sum_{1\le j < k \le m}(-a_ja_{k-1}-a_{j-1}a_k+a_{2j-k}a_{2k-j-1}+a_{2j-k-1}a_{2k-j})t \\
+ \sum_{1\le j < k \le m}(a_{j-1}a_{k-1}-a_{2j-k-1}a_{2k-j-1})
\end{gather*}
We conclude
\begin{equation}
H_{m-2,1,1}=\sum_{1\le j < k \le m} (-a_ja_{k-1}-a_{j-1}a_k+a_{2j-k}a_{2k-j-1}+a_{2j-k-1}a_{2k-j})
\end{equation}
So, we must prove the identity
\begin{gather}
\sum_{1\le j < k \le m} (-a_ja_{k-1}-a_{j-1}a_k+a_{2j-k}a_{2k-j-1}+a_{2j-k-1}a_{2k-j}) = \nonumber \\
=\sum_{1\le j < k \le m-1}(a_j a_{k} -
a_{2j-k} a_{2k-j})
-\sum_{1\le i \le m-1} a_i^2 - 2\sum_{1\le j < k \le m-1}a_ja_k \label{Hm-211}
\end{gather}
The right hand side of \eqref{Hm-211} is equal to
\begin{equation}
-\sum_{1\le i \le m-1} a_i^2-\sum_{1\le j < k \le m-1}a_{2j-k} a_{2k-j}
 - \sum_{1\le j < k \le m-1}a_ja_k
\end{equation}
We denote by
{\small
\begin{gather}
A_1=\sum_{1\le j < k \le m} -a_ja_{k-1}, \;
A_2=\sum_{1\le j < k \le m} -a_{j-1}a_k, \;
A_3= \sum_{1\le j < k \le m} a_{2j-k}a_{2k-j-1}, \;
A_4= \sum_{1\le j < k \le m} a_{2j-k-1}a_{2k-j}, \nonumber \\
B_1=-\sum_{1\le i \le m-1} a_i^2, \;
B_2 = -\sum_{1\le j < k \le m-1}a_{2j-k} a_{2k-j}
 \quad \text{ and } \; B_3=-\sum_{1\le j < k \le m-1}a_ja_k
\end{gather}}
We must show that
$(A_1+A_2+A_3+A_4) - (B_1 + B_2 + B_3) = 0.$\\
Because we have $A_1=B_1+B_3$, it rests to show \begin{equation} \label{eq:SarrafEq}A_2+A_3+A_4-B_2=0.
\end{equation}
For any $r,s$ with $(0\le r \le s \le m)$ we shall show that the coefficient at $a_ra_s$ in \eqref{eq:SarrafEq} is $0$.
\begin{itemize}
\item[Case 1: $r=s$.] \ We look for monomials $a_ra_s$, where $r=s$.
\begin{enumerate}
\item In $A_2$ we have $a_{j-1}a_k \notin \{a_0^2, \ldots , a_m^2\}$ for $j<k$.
\item  The monomials $a_{2j-k}a_{2k-j-1}$ in $A_3$ are not of the form $\{a_0^2, \ldots , a_m^2\}$, because the equality $2j-k=2k-j-1$ or equivalently $3j=3k-1$ cannot hold.
\item Analogously, the monomials $a_{2j-k-1}a_{2k-j}$ in $A_4$ are not of the form $\{a_0^2, \ldots , a_m^2\}$, because the equality $2j-k-1=2k-j$ or equivalently $3j-1=3k$ cannot hold.
\item  The monomials $a_{2j-k}a_{2k-j}$ in $B_2$ are not of the form $\{a_0^2, \ldots , a_m^2\}$, since $2j-k=2k-j$ or equivalently $3j=3k$ is impossible for $j<k$.
\end{enumerate}
We get that for all $r=0,\ldots , m$ the terms $a_r^2$ enter in equation \eqref{eq:SarrafEq} with coefficient $0$.
\item[Case 2: $r<s$.] \ We consider the set of all monomials $\{a_ra_s\vert 0\le r < s \le m\}$ and find their coefficients in $A_2, A_3, A_4$ and $B_2$.
\begin{enumerate}
\item Terms of $A_2$: We have $j-1<k$, because of $j<k$. Hence we put $j-1=r$ and $k=s$. Therefore in $A_2$ only  monomials $$\{a_ra_s\vert 0\le r < s-1 \le m-1\}=\{a_ra_s\vert 0\le r \le m-2, r<s-1, 2\le s \le m\}$$ appear with coefficient $-1$.
\item Terms of $A_3$: We have $2j-k < 2k-j-1$, because $j<k$. Hence we put $2j-k=r$ and $2k-j-1 =s$.
We get $3k-3j = s+1-r$, consequently $r\equiv s+1 \mod 3$. Furthermore, $j\le k-1$ implies $r\le s-2$. Hence in $A_3$ we have the monomials
$$\{a_ra_s\vert 0\le r < s-1 \le m-1, r\equiv s+1 \mod 3\}
$$ with coefficient $1$.
\item Terms of $A_4$: We have $2j-k-1< 2k-j$, because $j<k$. Hence we put $2j-k-1=r$ and $2k-j=s$. This gives $3k-3j = s-r-1$. By $j<k$ we get $r<s-1$ and $r\equiv s+2 \mod 3$. Therefore all the monomials In $A_4$ are of the form
$$\{a_ra_s\vert 0\le r < s-1 \le m-1, r\equiv s+2 \mod 3\}
$$ and have coefficient $1$.
\item Terms of $B_2$: We have $2j-k < 2k-j$, because $j<k$. Hence we put $2j-k=r$ and $2k-j=s$. This gives $3k-3j = s - r$.
We see $r \equiv s \mod 3$ and $r < s-1$. Because of $2k-j =s\le m$ and $j < m$, we have $k<m$. Therefore all the monomials in $B_2$ are of the form
$$\{a_ra_s\vert 0\le r < s-1 \le m-1, r\equiv s \mod 3\}
.$$ and have coefficient $-1$.
\end{enumerate}
We see, that $A_2+A_3+A_4 - B_2 = 0$ holds by using the monomials $a_ra_s, (r<s)$:
\begin{gather*}
A_2+A_3+A_4-B_2 = \\
\sum_{0\le r < s-1 \le m-1}-a_ra_s
+\sum_{\substack{0\le r < s-1 \le m-1 \\ r\equiv s+1 \mod 3}} a_ra_s
+\sum_{\substack{0\le r < s-1 \le m-1 \\ r\equiv s+2 \mod 3}} a_ra_s - \left(-\sum_{\substack{0\le r < s-1 \le m-1 \\ r\equiv s \mod 3}}a_ra_s \right) = 0
\end{gather*}

\end{itemize}
\end{proof}
\section{Formula for $H_{1,1,2}$.}
For $q=2$ and $n=2$ we have $k=m+1$.
\begin{mythm} \thlabel{EhbauerTheorem}
$H_{1,1,2} = -2a_0^2 D(m,1) - 2 (a_0+a_1)D(m,0).$
\label{thmn2}
\end{mythm}

We start with the matrix $\mathcal M(P,2,m+1)$, which has the form:
{
\begin{gather*}
\label{MP2mevenred}
\left(\begin{matrix}
a_{1}t^2 - 2a_0t &  a_{3}t^2-2a_2t+a_1&  a_5t^2-2a_4t+a_3 & \ldots & 0 \\
a_{0}t^2&  a_{2}t^2-2a_1t+a_0&  a_4t^2-2a_3t+a_2 & \ldots & 0 \\
0    &  a_{1}t^2 - 2a_0t &  a_{3}t^2-2a_2t+a_1& \ldots & 0 \\
0    &  a_0t^2  &  a_2t^2-2a_1t+a_0 & \ldots & 0 \\
\dots & \dots & \dots  & \ldots & \ldots   \\
0 & 0 & 0 & \ldots & 0\\
0 & 0 & 0 & \ldots & a_{m}  \\
0 & 0 & 0 &  \ldots & -2a_mt+a_{m-1}
\end{matrix} \right)
\end{gather*}}

For any square matrix $M$ of order $s\in \mathbb{N}$, let $M_l,$ $l\leq s,$ be  the submatrix of $M$ obtained by elimination of the $l$-th row and the $l$-th column and let $M_{i,j},$ $i\leq s, j\leq s,$ be the submatrix of $M$ obtained by elimination of the $i$-th row and the $j$-th column.
\\

For $l\in \{1,\ldots, m+1\}$ we consider the submatrix $\mathcal M(P,2,m+1)_l$ of $\mathcal M(P,2,m+1)$ to define the matrix $B(l,u), u \in \{1,\ldots, m\}$ as follows:
\\

The $c$-th column of $B(l,u)$ is the column of the $t$-free entries of the $c$-th column of $\mathcal M(P,2,m+1)_l,$ if $c\neq u,$ and $(-\frac{1}{2})$ $\cdot$ (column of the coefficients at $t$ of the $u$-th column of $\mathcal M(P,2,m+1)_l$), if $c = u$.
\\

The matrix $B(l,u)$ can also be obtained in the following way:
\\

There are 3 matrices ${\mathcal R}_0(P,2,m+1), {\mathcal R}_1(P,2,m+1), {\mathcal R}_2(P,2,m+1) \in M(k\times k,\mathbb{Z}[a_0,\ldots, a_m])$ such that $$\mathcal M(P,2,m+1) = {\mathcal R}_2(P,2,m+1)\cdot t^2 -2\cdot{\mathcal R}_1(P,2,m+1)\cdot t + {\mathcal R}_0(P,2,m+1).$$
Hence we have $$\mathcal M(P,2,m+1)_l = {\mathcal R}_2(P,2,m+1)_l\cdot t^2 -2\cdot{\mathcal R}_1(P,2,m+1)_l\cdot t + {\mathcal R}_0(P,2,m+1)_l.$$
Then for $c\neq u$ the $c$-th column of $B(l,u)$ is the $c$-th column of ${\mathcal R}_0(P,2,m+1)_l$ and the $u$-th column of $B(l,u)$ is the $u$-th column of ${\mathcal R}_1(P,2,m+1_l)$. If we denote $B(l,u) = (b_{i,j})_{\substack{1\leq i \leq m \\ 1 \leq j \leq m}},$ ${\mathcal R}_0(P,2,m+1)_l= (v_{i,j})_{\substack{1\leq i \leq m \\ 1 \leq j \leq m}}$ and ${\mathcal R}_1(P,2,m+1)_l =(s_{i,j})_{\substack{1\leq i \leq m \\ 1 \leq j \leq m}},$ we have
$b_{i,j} =
v_{i,j}$ for $j\neq u$ and $b_{i,u} = s_{i,u}.$
\begin{framed}
\begin{mytable}\label{table}
We give a detailed illustration by the example $m=3$.
{
\begin{itemize}
\item[]$\mathcal M(P,2,4)=
\left(\begin{matrix}
a_1t^2-2a_0t & a_3t^2-2a_2t+a_1 & a_3 & 0  \\
a_0t^2       & a_2t^2-2a_1t+a_0 & -2a_3t+a_2 & 0  \\
0& a_1t^2-2a_0t & a_3t^2-2a_2t+a_1 & a_3  \\
0& a_0t^2       & a_2t^2-2a_1t+a_0 & -2a_3t+a_2
\end{matrix}\right)$
\item[] $\mathcal M(P,2,4)_1=
\left(\begin{matrix}
a_2t^2-2a_1t+a_0 & -2a_3t+a_2 & 0  \\
a_1t^2-2a_0t & a_3t^2-2a_2t+a_1 & a_3  \\
a_0t^2       & a_2t^2-2a_1t+a_0 & -2a_3t+a_2
\end{matrix}\right)$
\item[]
$B(1,1)=
\left(\begin{matrix}
a_1 & a_2 & 0  \\
a_0 & a_1 & a_3  \\
0       & a_0 & a_2
\end{matrix}\right), \quad
B(1,2)=
\left(\begin{matrix}
a_0 & a_3 & 0  \\
0 & a_2 & a_3  \\
0  & a_1 & a_2
\end{matrix}\right), \quad
B(1,3)=
\left(\begin{matrix}
a_0 & a_2 & 0  \\
0 & a_1 & 0  \\
0       & a_0 & a_3
\end{matrix}\right), \quad$
\item[] $\mathcal M(P,2,4)_2=
\left(\begin{matrix}
a_1t^2-2a_0t & a_3 & 0  \\
0&  a_3t^2-2a_2t+a_1 & a_3  \\
0&  a_2t^2-2a_1t+a_0 & -2a_3t+a_2
\end{matrix}\right)$
\item[]
$B(2,1)=
\left(\begin{matrix}
a_0 & a_3 & 0  \\
0 & a_1 & a_3  \\
0       & a_0 & a_2
\end{matrix}\right), \quad
B(2,2)=
\left(\begin{matrix}
0 & 0 & 0  \\
0 & a_2 & a_3  \\
0  & a_1 & a_2
\end{matrix}\right), \quad
B(2,3)=
\left(\begin{matrix}
0 & a_3 & 0  \\
0 & a_1 & 0  \\
0       & a_0 & a_3
\end{matrix}\right), \quad$
\item[] $\mathcal M(P,2,4)_3=
\left(\begin{matrix}
a_1t^2-2a_0t & a_3t^2-2a_2t+a_1 & 0  \\
a_0t^2       & a_2t^2-2a_1t+a_0 & 0  \\
0& a_0t^2       &-2a_3t+a_2
\end{matrix}\right)$
\item[]
$B(3,1)=
\left(\begin{matrix}
a_0 & a_1 & 0  \\
0 & a_0 & a_3  \\
0       & 0 & a_2
\end{matrix}\right), \quad
B(3,2)=
\left(\begin{matrix}
0 & a_2 & 0  \\
0 & a_1 & 0  \\
0  & 0 & a_2
\end{matrix}\right), \quad
B(3,3)=
\left(\begin{matrix}
0 & a_1 & 0  \\
0 & a_0 & 0  \\
0       & 0 & a_3
\end{matrix}\right), \quad$
\item[]
$\mathcal M(P,2,4)_4=
\left(\begin{matrix}
a_1t^2-2a_0t & a_3t^2-2a_2t+a_1 & a_3   \\
a_0t^2       & a_2t^2-2a_1t+a_0 & -2a_3t+a_2  \\
0& a_1t^2-2a_0t & a_3t^2-2a_2t+a_1
\end{matrix}\right)$
\item[]
$B(4,1)=
\left(\begin{matrix}
a_0 & a_1 & a_3  \\
0 & a_0 & a_2  \\
0       & 0 & a_1
\end{matrix}\right), \quad
B(4,2)=
\left(\begin{matrix}
0 & a_2 & a_3 \\
0 & a_1 & a_2  \\
0  & a_0 & a_1
\end{matrix}\right), \quad
B(4,3)=
\left(\begin{matrix}
0 & a_1 & 0  \\
0 & a_0 & a_3  \\
0       & 0 & a_2
\end{matrix}\right), \quad$
\end{itemize}}
\end{mytable}
\end{framed}
\begin{myproposition}
\label{PropH112}
$$H_{1,1,2}= -2\cdot \sum_{\substack{l=1,\ldots, m+1 \\ u = \{1,\ldots ,m\}}} \det B(l,u). $$
\end{myproposition}
\begin{proof}
The definition of $H_{1,1,2}$ and the construction of $B(l,u)$ give the identity.
\end{proof}
\begin{mylemma} \label{Blu}
If $l\neq 1$ and $u\neq 1$, then $\det B(l,u) = 0$.
\end{mylemma}
\begin{proof}
If $l\neq 1$ and $u\neq 1$, then all entries of the first column of $B(l,u)$ are $0$.
Indeed the entries of the first column of $\mathcal M(P,2,m+1)$ have no degree $0$ term in $t$. These entries are $a_{2 - i}t^2 - 2a_{1 -i }t+a_{-i},$ where $i\in \{1, \ldots , m+1\}$ is the row-index and  $a_{-i}=0$ for $i\in \{1, \ldots , m+1\}$.
Therefore,  if $l\neq 1$, then the entries of the first column of the submatrix $\mathcal M(P,2,m+1)_l$ have no degree $0$ term in $t$. Consequently by construction the entries of the first column of $B(l,u)$ are $0$ for $l\neq 1, u\neq 1.$
\end{proof}
The above table \ref{table} illustrates this lemma.
\begin{mylemma} \label{B21}
We have $\det B(2,1) = a_0\cdot D(m,0)$.
\end{mylemma}
\begin{proof}
The matrix $B(2,1)$ has the form:
$$B(2,1)=\left(
\begin{matrix}
a_0 & a_3 & a_5 & a_7 &  \ldots & 0 & 0 & 0 \\
0   & a_1 & a_3 & a_5 &  \ldots& 0 & 0 & 0  \\
0   & a_0 & a_2 & a_4 &  \ldots & 0 & 0 & 0 \\
0& 0   &  a_1&a_3 &  \ldots & 0& 0 & 0 \\
0& 0   & a_0&a_2 &  \ldots&0&  0 & 0\\
\ldots   & \ldots   &   \ldots  &\ldots  & \ldots & \ldots & \ldots & \ldots \\
0   & 0   &  0  &  0  & \ldots&a_{m-3}  & a_{m-1}  &0  \\
0& 0& 0& 0&  \ldots & a_{m-4} & a_{m-2} & a_{m}\\
0   & 0   &  0  &  0  & \ldots&a_{m-5} & a_{m-3} & a_{m-1}  \\
\end{matrix}
 \right)
 $$
As the submatrix $B(2,1)_{1}$ is equal to the matrix $\mathcal M_{nt}(P,0,m),$ this gives
$$\det B(2,1) = a_0 D(m,0)$$
\end{proof}

\begin{mylemma} \label{LemmaBl1}
$$\sum_{l=3}^{m+1}\det B(l,1)=a_0^2D(m,1).$$
\end{mylemma}
\begin{proof}
Let $l\ge 3.$ The matrix $B(l,1)$ has the form:
\\
$$B(l,1)=
\begin{array}{c}
	\begin{array}{ccccccccc}
	&&&\multicolumn{4}{c}{\text{$l$-th column eliminated}} &&
	\end{array}
	\\
	\left( \begin{array}{cc|ccc|cccc}
a_0 & a_1 & a_3 & a_5  & \ldots &\ldots& 0 & 0 & 0 \\
0   & a_0 & a_2 & a_4 & \ldots&\ldots& 0 & 0 & 0  \\ \hline
0   & 0 & a_1 & a_3 &  \ldots &\ldots & 0 & 0 & 0 \\
0& 0   &  a_0&a_2 & \ldots& \ldots & 0& 0 & 0 \\
0& 0   & 0&a_1 & \ldots&\ldots&0&  0 & 0\\
\ldots   & \ldots  & \ldots  & \ldots &\ldots & \ldots & \ldots & \ldots & \ldots
\\
\hline
\multicolumn{4}{c}{\text{$l$-th row eliminated}}&&&&&
\\
\ldots   & \ldots  & \ldots  & \ldots  & \ldots & \ldots & \ldots & \ldots & \ldots
\\
0   & 0   &  0  &  0  &  \ldots&\ldots &a_{m-3}  & a_{m-1}  &0  \\
0& 0& 0& 0&  \ldots & \ldots & a_{m-4} & a_{m-2} & a_{m}\\
0   & 0   &  0  &  0  & \ldots & \ldots &a_{m-5} & a_{m-3} & a_{m-1}  \\
\end{array}
\right)
\end{array}$$
The submatrix $(B(l,1)_1)_1$ obtained by elimination of the first two rows and the first two columns of $B(l,1)$ is equal to the submatrix $\mathcal M_{nt}(P,0,m)_{l-2}$:
$$\mathcal M_{nt}(P,0,m)_{l-2}=
\begin{array}{c}
\begin{array}{ccccccc}
&\multicolumn{4}{c}{\text{$(l-2)$-th column eliminated}} &&
\end{array}\\
\left(
\begin{array}{ccc|cccc}
a_1 & a_3 &  \ldots &\ldots & 0 & 0 & 0 \\
a_0&a_2 & \ldots& \ldots & 0& 0 & 0 \\
0&a_1 & \ldots&\ldots&0&  0 & 0\\
0 & a_0 & \ldots & \ldots & 0 &0 & 0\\
\ldots   &   \ldots &\ldots & \ldots & \ldots & \ldots & \ldots
\\
\hline
\multicolumn{4}{c}{\text{$(l-2)$-th row eliminated}}&&&
\\
\ldots  &  \ldots  &   \ldots & \ldots & \ldots & \ldots & \ldots
\\
0  &  0  &  \ldots&\ldots &a_{m-3}  & a_{m-1}  &0  \\
0& 0&  \ldots & \ldots & a_{m-4} & a_{m-2} & a_{m}\\
0  &  0  &  \ldots & \ldots &a_{m-5} & a_{m-3} & a_{m-1}  \\
\end{array}
\right)
\end{array}$$
The minors $\det \mathcal M_{nt}(P,0,m)_l$ are exactly the summands of the coefficient of the term of degree 1 in $(-U)$ of $\det (\mathcal M_{nt}(P,0,m)-I_mU)$. This results in
$$\sum_{3\leq l \leq m+1} \det \mathcal M_{nt}(P,0,m)_{l-2} = D(m,1)$$
Expanding the determinant of $B(l,1)$ along the first two columns gives:
$$\det B(l,1) = a_0^2\cdot \det \mathcal M_{nt}(P,0,m)_{l-2}$$
We conclude
$$\sum_{l=3}^{m+1} \det B(l,1) = \sum_{l=3}^{m+1} a_0^2\cdot \det \mathcal M_{nt}(P,0,m)_{l-2} = a_0^2\cdot D(m,1)$$
\end{proof}

\begin{mylemma} \label{B11}
$$\det B(1,1) = a_1 \cdot D(m,0) -a_0\det B(1,1)_{2,1}.$$
\end{mylemma}
\begin{proof}
Let us consider the matrix $B(1,1)$:
$$B(1,1)=\left(
\begin{matrix}
a_1 & a_2 & a_4 & a_6 &  \ldots & 0 & 0 & 0 \\
a_0   & a_1 & a_3 & a_5 &  \ldots& 0 & 0 & 0  \\
0   & a_0 & a_2 & a_4 &  \ldots & 0 & 0 & 0 \\
0& 0   &  a_1&a_3 &  \ldots & 0& 0 & 0 \\
0& 0   & a_0&a_2 &  \ldots&0&  0 & 0\\
\ldots   &  \ldots  &   & \ldots & \ldots & \ldots & \ldots & \ldots \\
0   & 0   &  0  &  0  & \ldots&a_{m-3}  & a_{m-1}  &0  \\
0& 0& 0& 0& \ldots & a_{m-4} & a_{m-2} & a_{m}\\
0   & 0   &  0  &  0  & \ldots&a_{m-5} & a_{m-3} & a_{m-1}  \\
\end{matrix}
 \right)
 $$
We see $B(1,1)_{1}= \mathcal M_{nt}(P,0,m)$. Expanding the determinant of $B(1,1)$ along the first column immediately gives the result:
$$\det B(1,1) = a_1\cdot \det \mathcal M_{nt}(P,0,m) - a_0\cdot \det {B}(1,1)_{2,1}=a_1\cdot D(m,0) - a_0\cdot \det {B}(1,1)_{2,1} .$$
\end{proof}

We consider $B(1,u)$ for $u\ge 2.$ By the definition of $\mathcal M(P,2,m+1)$ and the construction of $B(1,u)$ we have:
$$B(1,u)=
\begin{array}{c}
\begin{array}{ccccccccc}
&&\text{$u$-th column}&&&&&
\end{array}
\\
\left(
\begin{matrix}
a_0 & \ldots &a_{2u-4} & a_{2u-1} & a_{2u}& \ldots & 0 &  0 \\
0   &\ldots &a_{2u-5} & a_{2u-2} & a_{2u-1}& \ldots & 0 & 0 \\
0   & \ldots &a_{2u-6} & a_{2u-3} & a_{2u-2}&\ldots & 0 &  0 \\
0   & \ldots &a_{2u-7} & a_{2u-4} & a_{2u-3}&\ldots & 0 &  0 \\
\ldots   &  \ldots & \ldots & \ldots &\ldots &   \ldots & \ldots &  \ldots \\
0   & \ldots &a_{2u-m-1} & a_{2u-(m-2)} & a_{2u-(m-3)}& \ldots  & a_{m-1}  &0  \\
0& \ldots &a_{2u-m-2} & a_{2u-(m-1)} & a_{2u-(m-2)}&\ldots &  a_{m-2} & a_{m}\\
0   &\ldots &a_{2u-m-3} & a_{2u-m} & a_{2u-(m-1)}& \ldots& a_{m-3} & a_{m-1}  \\
\end{matrix}
 \right)
\end{array}$$
The Laplace expansion of the determinant of $B(1,u)$ along the first column gives the following result:

\begin{mylemma} \label{B1u}
$$\det B(1,u) = a_0 \cdot \det B(1,u)_1\;, \quad \text{ for }u=2, \ldots, m$$
\newline
\end{mylemma}

\begin{proof}[Proof of theorem \ref{thmn2}:]
\mbox{}\\
By proposition \ref{PropH112} we have
$$H_{1,1,2}= -2\cdot \sum_{\substack{l=1,\ldots, m+1 \\ u = \{1,\ldots ,m\}}} \det B(l,u) .$$
By lemma \ref{Blu} we have $\det B(l,u) = 0$ for $l \in \{2, \ldots, m+1\}, u \in \{2,\ldots, m\}$.
\\
By lemma \ref{B21} we have $\det B(2,1) = a_0\cdot D(m,0).$
\\
Lemma \ref{LemmaBl1} gives the equality
$$\sum_{l=3}^{m+1}\det B(l,1)=a_0^2D(m,1).$$
Lemma \ref{B11} gives the equality
$$\det B(1,1) = a_1 \cdot D(m,0) -a_0\det B(1,1)_{2,1}.$$
By lemma \ref{B1u} we have for $u\in \{2, \ldots, m\}$ $$\det B(1,u) = a_0 \cdot \det B(1,u)_1.$$
This gives the result:
$$H_{1,1,2}=-2\cdot (a_0^2\cdot D(m,1) + a_0 \cdot D(m,0) + a_1 \cdot D(m,0)-a_0\det B(1,1)_{2,1}+a_0\sum_{u=2}^m\det B(1,u)_1)$$
Therefore, to prove the theorem, it remains to show that
\begin{equation} \label{principaleq}\det B(1,1)_{2,1}=\sum_{u=2}^m\det B(1,u)_1 \quad .
\end{equation}
\renewcommand{\qedsymbol}{}
\end{proof}
\renewcommand{\qedsymbol}{$\square$}
\noindent
We give an explicit example of the equality \ref{principaleq} for $m=4$.
To explain this equality, we start with the matrix
$\mathcal M_{nt}(P,0,4)= \begin{pmatrix}
a_1&a_3&a_5 \\
a_0&a_2&a_4\\
a_{-1}&a_1&a_3
\end{pmatrix}$, where $a_i = 0$ for $i\notin \{a_0,\ldots , a_4\}$.
If we consecutively increase the row-indices  of the entries of the first, second and third row of $M_{nt}(P,0,4)$ by 1, we get three matrices $B(1,1)_{2,1}=\begin{pmatrix}
a_2&a_4&a_6 \\
a_0&a_2&a_4\\
a_{-1}&a_1&a_3
\end{pmatrix} =
\begin{pmatrix}
a_2&a_4&0 \\
a_0&a_2&a_4\\
0&a_1&a_3
\end{pmatrix},\quad \begin{pmatrix}
a_1&a_3&a_5 \\
a_1&a_3&a_5\\
a_{-1}&a_1&a_3
\end{pmatrix}, \quad
\begin{pmatrix}
a_1&a_3&a_5 \\
a_0&a_2&a_4\\
a_0&a_2&a_4
\end{pmatrix}$.
The determinant of the last two matrices is $0$, because they have two identical lines. Therefore, if we build the sum of the determinants, we get the left side of the equality \eqref{principaleq}
$$\begin{array}{ccccccc}
\det B(1,1)_{2,1}&=&\det B(1,2)_1&+&\det B(1,3)_1&+&\det B(1,4)_1 \\
\begin{vmatrix}
a_2&a_4&0 \\
a_0&a_2&a_4\\
0&a_1&a_3
\end{vmatrix}
&=&
\begin{vmatrix}
a_2&a_3&0 \\
a_1&a_2&a_4\\
a_0&a_1&a_3
\end{vmatrix}
&+&
\begin{vmatrix}
a_1&a_4&0 \\
a_0&a_3&a_4\\
0&a_2&a_3
\end{vmatrix}
&+&
\begin{vmatrix}
a_1&a_3&0 \\
a_0&a_2&0\\
0&a_1&a_4
\end{vmatrix}
\end{array}$$
If we consecutively increase the column-indices  of the entries of the first, second and third column of $M_{nt}(P,0,4)$ by 1, we get the three matrices $B(1,2)_1, B(1,3)_1$ and $B(1,4)_1$. The sum of the determinants gives the right side of \eqref{principaleq}.

To prove equality \eqref{principaleq}, we need the following:
\\

For $i,j,k\in\{1, \ldots , n\}$ let $\alpha _{i,j,k}$ be generators of the free ring $\mathbb{Z}[\alpha _{i,j,k}]$.
Let $\sigma \in S_n,$ where $S_n$ is the symmetric group of order $n$. We define:
{$$\displaystyle M(r,\sigma):= (\alpha_{i,j,{\sigma(i)}})_{\substack{1 \leq i \leq n \\ 1 \leq j \leq n}}\quad \text{ and } \quad
M(c,\sigma):=(\alpha_{i,j,{\sigma(j)}})_{\substack{1 \leq i \leq n \\ 1 \leq j \leq n}}$$}
($r$, $c$ mean rows, columns respectively).
\begin{myexample2}
Let $n=3$ and $\sigma = \left(\begin{matrix} 1&2&3 \\ 2 & 3 & 1 \end{matrix}\right)$. We have
$$M(r,\sigma)=\begin{pmatrix}
\alpha_{1,1,\sigma(1)}&\alpha_{1,2,\sigma(1)}&\alpha_{1,3,\sigma(1)} \\
\alpha_{2,1,\sigma(2)}&\alpha_{2,2,\sigma(2)}&\alpha_{2,3,\sigma(2)} \\
\alpha_{3,1,\sigma(3)}&\alpha_{3,2,\sigma(3)}&\alpha_{3,3,\sigma(3)}
\end{pmatrix}
= \begin{pmatrix}
\alpha_{1,1,2}&\alpha_{1,2,2}&\alpha_{1,3,2} \\
\alpha_{2,1,3}&\alpha_{2,2,3}&\alpha_{2,3,3} \\
\alpha_{3,1,1}&\alpha_{3,2,1}&\alpha_{3,3,1}
\end{pmatrix}$$
$$
M(c,\sigma)=\begin{pmatrix}
\alpha_{1,1,\sigma(1)}&\alpha_{1,2,\sigma(2)}&\alpha_{1,3,\sigma(3)} \\
\alpha_{2,1,\sigma(1)}&\alpha_{2,2,\sigma(2)}&\alpha_{2,3,\sigma(3)} \\
\alpha_{3,1,\sigma(1)}&\alpha_{3,2,\sigma(2)}&\alpha_{3,3,\sigma(3)}
\end{pmatrix}
= \begin{pmatrix}
\alpha_{1,1,2}&\alpha_{1,2,3}&\alpha_{1,3,1} \\
\alpha_{2,1,2}&\alpha_{2,2,3}&\alpha_{2,3,1} \\
\alpha_{3,1,2}&\alpha_{3,2,3}&\alpha_{3,3,1}
\end{pmatrix}$$
\end{myexample2}

\begin{myproposition}\label{DetIdent}
$$ \sum_{\sigma \in S_n} \det M(r,\sigma) =
\sum_{\sigma \in S_n} \det M(c,\sigma)
$$

\end{myproposition}
\begin{proof}
\footnote{\thanks{Darij Grinberg; CSENG School of Mathematics, University of Minnesota}}
For a fixed $\sigma \in S_n$ the determinant of the matrix $M(r,\sigma)=(\alpha _{i,j,{\sigma (i)}})_{\substack{1 \leq i \leq n \\ 1 \leq j \leq n}}$ is :
$$\det (\alpha _{i,j,{\sigma (i)}})_{\substack{1 \leq i \leq n \\ 1 \leq j \leq n}}= \sum_{\rho \in S_n} (-1)^{sgn(\rho)}\alpha_{1,\rho(1),{\sigma(1)}}\cdots \alpha _{n,\rho(n),{\sigma (n)}}$$
We build the sum over $\sigma \in S_n:$
$$\sum_{\sigma \in S_n} \det (\alpha _{i,j,{\sigma (i)}})_{\substack{1 \leq i \leq n \\ 1 \leq j \leq n}}=\sum_{\sigma \in S_n} \sum_{\rho \in S_n} (-1)^{sgn(\rho)}\alpha_{1,\rho(1),{\sigma(1)}}\cdots \alpha _{n,\rho(n),{\sigma (n)}}=$$
\begin{equation}
\sum_{\rho \in S_n} \sum_{\sigma \in S_n} (-1)^{sgn(\rho)}\alpha_{1,\rho(1),{\sigma(1)}}\cdots \alpha _{n,\rho(n),{\sigma (n)}}
\label{eq:prp1}
\end{equation}
In the same way we consider the index of the columns $(\alpha _{i,j,{\sigma (j)}})_{\substack{1 \leq i \leq n \\ 1 \leq j \leq n}}$ with $\sigma \in S_n$:
$$\det (\alpha _{i,j,{\sigma (j)}})_{\substack{1 \leq i \leq n \\ 1 \leq j \leq n}}= \sum_{\rho \in S_n} (-1)^{sgn(\rho)}\alpha_{1,\rho(j),{\sigma(\rho(j))}}\cdots \alpha _{n,\rho(n),{\sigma (\rho(n))}}$$
We build the sum over $\sigma \in S_n:$
$$\sum_{\sigma \in S_n} \det (\alpha _{i,j,{\sigma (j)}})_{\substack{1 \leq i \leq n \\ 1 \leq j \leq n}}=\sum_{\sigma \in S_n} \sum_{\rho \in S_n} (-1)^{sgn(\rho)}\alpha_{1,\rho(1),{\sigma(\rho(1))}}\cdots \alpha _{n,\rho(n),{\sigma (\rho(n))}}=$$
\begin{equation}
\sum_{\rho \in S_n} \sum_{\sigma \in S_n} (-1)^{sgn(\rho)}\alpha_{1,\rho(1),{\sigma(\rho(1))}}\cdots \alpha _{n,\rho(n),{\sigma (\rho(n))}}
\label{eq:prp2}
\end{equation}
For a fixed permutation $\rho \in S_n$ in equation \eqref{eq:prp1} and \eqref{eq:prp2} we get two expressions:
$$\sum_{\sigma \in S_n} (-1)^{sgn(\rho)}\alpha_{1,\rho(1),{\sigma(1)}}\cdots \alpha _{n,\rho(n),{\sigma (n)}}$$ and $$\sum_{\sigma \in S_n} (-1)^{sgn(\rho)}\alpha_{1,\rho(1),{\sigma(\rho(1))}}\cdots \alpha _{n,\rho(n),{\sigma (\rho(n))}}$$.
\\
As $\sigma$ runs over all permutations in $S_n$, so does $\sigma \circ \rho$, that is $\{\sigma \vert \sigma \in S_n\} = \{ \sigma \circ \rho \vert \sigma \in S_n\}$. We conclude:
$$\sum_{\sigma \in S_n} (-1)^{sgn(\rho)}\alpha_{1,\rho(1),{\sigma(1)}}\cdots \alpha _{n,\rho(n),{\sigma (n)}} = \sum_{\sigma \in S_n} (-1)^{sgn(\rho)}\alpha_{1,\rho(1),{\sigma(\rho(1))}}\cdots \alpha _{n,\rho(n),{\sigma (\rho(n))}}$$ and
$$\sum_{\rho \in S_n} \sum_{\sigma \in S_n} (-1)^{sgn(\rho)}\alpha_{1,\rho(1),{\sigma(1)}}\cdots \alpha _{n,\rho(n),{\sigma (n)}}=
\sum_{\rho \in S_n} \sum_{\sigma \in S_n} (-1)^{sgn(\rho)}\alpha_{1,\rho(1),{\sigma(\rho(1))}}\cdots \alpha _{n,\rho(n),{\sigma (\rho(n))}} $$
\end{proof}

\noindent
We apply proposition \ref{DetIdent} to the case $n=m-1$,
\begin{equation}\label{DefOfAlphas}
\alpha_{i,j,k} =
\begin{cases} a_{2j-i}, \text{ if } k\neq 1,\; \\
a_{2j-i+1}, \text{ if } k = 1
\end{cases}
\end{equation}

\begin{mylemma}
\label{remarkMrsigma}
Let $\sigma \in S_{m-1}$.
\begin{enumerate}
\item If $\sigma(1)=1$, then $M(r,\sigma) = B(1,1)_{2,1}$.
\item If $\sigma(1)\neq 1$, then $\det M(r,\sigma)=0$.
\item For $v\in \{1,\ldots, m-1\}$ we have: If $\sigma(v)=1,$ then $M(c,\sigma)= B(1,v+1)_1$, \\
or equivalently for all $\sigma \in S_{m-1}$  $M(c,\sigma)= B(1,\sigma^{-1}(1)+1)_1$ .
\end{enumerate}
\end{mylemma}
\begin{proof}
Let $u \in \{1, \ldots , m\}$. We have:
$$(i,j)\text{-th entry } b_{ij} \text{ of } B(1,u)  =
\begin{cases}
a_{2(j+1)-(i+1)-2}=a_{2j-i-1}, &\text{ if } j \neq u \\
a_{2(u+1)-(i+1)-1}=a_{2u-i}, &\text{ if } j=u
\end{cases}
$$
\begin{enumerate}
\item
Since $B(1,1)_{2,1}$ is the $(2,1)$ minor of $B(1,1)$, we have:
$$(i,j)\text{-th entry of } B(1,1)_{2,1} =
\begin{cases} b_{1,j+1} = a_{2j}  &\text{ for } i=1, 1\le j \le m-1 \\
b_{i+1,j+1}=a_{2j-i}  &\text{ for } 2 \le i \le m-1, 1 \le j \le m-1
\end{cases}$$
If $\sigma(1)=1$, then by definition \ref{DefOfAlphas} the entries $\alpha_{i,j,\sigma(i)}$ of the matrix $M(r,\sigma)$ are:
$$\alpha_{i,j,\sigma(i)} =
\begin{cases} a_{2j-i+1} = a_{2j}\quad, & \text{ if } i = 1 \\
a_{2j-i} \quad, &\text{ if } \sigma(i)\neq 1 \text{ equivalently } i \neq 1
\end{cases}$$
Therefore, the entries of $M(r,\sigma)$ and $B(1,1)_{2,1}$ are identical for $\sigma(1)=1.$
\item If $\sigma (1) \neq 1,$ then $v \in \{2, \ldots ,m-1\}$ exists  with $\sigma (v) =1.$ We have, that the $v$-th row and $(v-1)$-th row of $M(r,\sigma)$ are identical, since
$$\alpha_{v,j,\sigma(v)}=a_{2j-v+1} =a_{2j-(v-1)+0}=\alpha_{v-1,j,\sigma(v-1)} \quad \text{ for } j=1, \ldots, m-1.$$
We conclude $\det M(r, \sigma)=0.$
\item Let $u = v+1 \in \{2, \dots, m\}$.
As in Lemma \ref{B1u} defined, $B(1,u)_1$ is the $(1,1)$ minor of $B(1,u),$ where
$$(i,j)\text{-th entry of } B(1,u)_{1} = b_{i+1,j+1}=
\begin{cases}
a_{2(j+2)-(i+2)-2}=a_{2j-i}\quad, &\text{ if } j \neq v \\
a_{2(u+1)-(i+2)-1}=a_{2u-i-1}\quad, &\text{ if } j=v
\end{cases}.$$
If $\sigma (v)=1$, then the entries of the $i$-th row of $M(c,\sigma)$ are equal to:
$$\alpha_{i,j,{\sigma (j)}} = \begin{cases} a_{2j-i} & \text{ for } j\in \{1, \ldots , m-1\}, j\neq v \\
a_{2v-i+1}=a_{2u-i-1} & \text{ for } j=v \end{cases},$$ which are precisely the entries of $B(1,u)_1.$
\qedhere
\end{enumerate}
\end{proof}
Hence, we get:
\begin{enumerate}
\item $\sum_{\sigma \in S_{m-1}} \det M(r,\sigma) =(m-2)! \cdot\det B(1,1)_{2,1}$
\item $\sum_{\sigma \in S_{m-1}} \det M(c,\sigma) =(m-2)!\cdot\sum_{u=2}^m\det B(1,u)_1$
\end{enumerate}
Applying proposition \ref{DetIdent} to $\alpha_{i,j,k}$ from \ref{DefOfAlphas}, we get from the above formulas:

$$\det B(1,1)_{2,1}=\sum_{u=2}^m\det B(1,u)_1 $$ i. e. equality \ref{principaleq}, which concludes the proof of theorem \ref{thmn2}
\qed{}

\end{document}